\documentclass[abstracton,a4paper,10pt]{scrartcl}
\usepackage[utf8]{inputenc}
\usepackage{amsmath,amssymb,amsthm}
\usepackage{enumerate}
\usepackage{url}
\usepackage{color}
\usepackage{cite}
\newtheorem{thm}{Theorem}[section]
\newtheorem{dfn}[thm]{Definition}
\newtheorem{lem}[thm]{Lemma}
\newtheorem{prop}[thm]{Proposition}

\theoremstyle{remark}

\DeclareMathOperator{\scal}{Scal}
\DeclareMathOperator{\ric}{Ric}
\DeclareMathOperator{\Lop}{L}

\DeclareMathOperator{\vol}{vol}

\DeclareMathOperator{\Id}{I}
\DeclareMathOperator{\id}{id}

\newcommand{\R}{\mathrm{R}}

\newcommand{\ps}[2]{\left\langle #1,#2 \right\rangle}
\newcommand{\eps}{\varepsilon}

\newcommand{\ph}{\varphi}
\begin{document}

\title{Stability of nonnegative isotropic curvature under continuous
  deformations of the metric}

\author{Thomas Richard\footnote{Université Paris-Est, Laboratoire
    d'Analyse et de Mathématiques Appliquées (UMR 8050), UPEM, UPEC,
    CNRS, F-94010, Créteil, France.\newline
  \emph{e-mail:} thomas.richard@u-pec.fr}}
\maketitle
\begin{abstract}
  Using a method introduced by R. Bamler to study the behavior of
  scalar curvature under continuous deformations of Riemannian metrics, we prove that if a
  sequence of smooth Riemannian metrics $g_i$ on a fixed compact
  manifold $M$ has isotropic curvature bounded from below by a
  nonnegative function $\underline{u}$, and if $g_i$
  converge in $C^0$ norm to a smooth metric $g$, then $g$ has
  isotropic curvature bounded from below by $\underline{u}$. The proof
  also works for various other bounds from below on the curvature,
  such has non-negative curvature operator.
\end{abstract}

A major trend in modern Riemannian geometry is to understand 
the geometry of a metric space $(X,d)$ which is a limit of a sequence of smooth
manifolds $(M_i,g_i)_{i\in\mathbb{N}}$ in terms of given geometric
information on the elements of the sequence $(M_i,g_i)_{i\in\mathbb{N}}$. In a lot of
cases $(M_i,g_i)_{i\in\mathbb{N}}$ is supposed to converge to $(X,d)$
in the Gromov--Hausdorff (abbreviated as GH) sense. In this context one can show the
following results:
\begin{itemize}
\item If $(M_i,g_i)_{i\in\mathbb{N}}$ GH converges to a smooth
  manifold $(M,g)$ and each $(M_i,g_i)$ has sectionnal curvature
  greater than $K$, then $(M,g)$ has sectionnal curvature greater than
  $K$. This follows from the synthetic theory of lower sectionnal
  curvature bounds as exposed in \cite{MR1835418}.
\item If $(M^n_i,g_i)_{i\in\mathbb{N}}$ GH converges to a smooth
  manifold $(M^n,g)$ and each $(M^n_i,g_i)$ has Ricci curvature
  greater than $K$, then $(M^n,g)$ has Ricci curvature greater than
  $K$. Note that collapsing is ruled out in this case. The result is a
  consequence of the theory of $CD(K,N)$ spaces (see \cite{MR2480619}
  and \cite{MR2118836}).
\end{itemize}

The case of lower bounds on the scalar curvature is much less
clear. As the example of the product of a radius $1/i$ $2$-sphere with a negative scalar curvature manifold shows, a result
as above cannot be true under the sole assumption of Gromov-Hausdorff
convergence, at least if collapsing is allowed. Example 6.1 in the
preprint \cite{2017arXiv170300984B} is actually a 3 dimensional singular limit of a
sequence of 3-manifolds with positive scalar curvature which cannot be
reasonably thought of as having positive scalar curvature.
 
In \cite{MR3201312},
Gromov layed the first bricks of a possible synthetic theory of lower bounds on
scalar curvature and proved :
\begin{thm}[Gromov \cite{MR3201312}]
  Let $M$ be a smooth manifold, $\underline{u}:M\to\mathbb{R}$ be a
  continuous function, and $(g_i)_{i\in\mathbb{N}}$ be a
  sequence of possibly non-complete $C^2$ Riemannian metrics. Assume that:
  \begin{itemize}
  \item For every $i\in\mathbb{N}$, $\scal_{g_i}(x)\geq \underline{u} (x)$.
  \item $g_i$ converges in $C^0$ norm to a $C^2$ Riemannian metric
    $g$. 
  \end{itemize}
  Then $\scal_g(x)\geq \underline{u}(x)$.
\end{thm}
Here and in the rest of the paper, by ``$g_i$ converges in $C^0$ norm
to $g$'' we will mean that $\sup_M|g_i-g|$ goes to zero as $i$ goes to
infinity where the pointwise norm of the $2$-tensors $g_i-g$ is taken
with respect to a fixed background Riemannian metric.

Gromov's proof uses arguments from the theory of minimal
hypersurfaces. In \cite{zbMATH06609371}, Bamler gave an alternative
proof of this result using Ricci flow. In this paper we will adapt
Bamler's method to other type of curvature lower bounds and show the following theorem:

\begin{thm}\label{thm1}
  Let $M$ be a compact smooth manifold and $\underline{u}$ be a
  lower semi-continuous nonnegative function on $M$. Let $g_i$ be a sequence of
  complete smooth metrics with bounded curvature on $M$ which
  converges in $C^0$ norm to a smooth metric $g$. Then:
  \begin{itemize}
  \item if for every $i\in\mathbb{N}$ the lowest eigenvalue of the curvature operator of $g_i$
    at any $x\in M$ is bigger than $\underline{u}(x)$, then the lowest eigenvalue of the curvature operator of $g$
    at any $x\in M$ is bigger than $\underline{u}(x)$.
  \item if for every $i\in\mathbb{N}$ the isotropic curvature of $g_i$
    at any $x\in M$ is bigger than $\underline{u}(x)$, then the
    isotropic curvature\footnote{We say that $(M,g)$ has isotropic
      curvature greater than $k$ at $x\in M$ if, for any orthonormal $4$-frame
      $(e_1,\dots,e_4)$ of $T_xM$,
      $R(e_1,e_3,e_1,e_3)+R(e_1,e_4,e_1,e_4)+R(e_2,e_3,e_2,e_3)
      +R(e_2,e_4,e_2,e_4)-2 R(e_1,e_2,e_3,e_4)\geq k$. See Appendix
      $B$ for more details. } of $g$
    at any $x\in M$ is bigger than $\underline{u}(x)$.
  \end{itemize}
  Then $\R_g -\underline{u}\Id\in\mathcal{C}$.
\end{thm}
This will actually follow from the more general Theorem
\ref{thm:mainthm} below.

To state it, we need to recall some notation: let $A$ and
$B$ be two symmetric tensors on an inner product space $E$, we define
the algebraic curvature tensor $A\wedge B:\Lambda^2E\to\Lambda^2E$ by:
\[(A\wedge B)(x\wedge y)=\tfrac{1}{2}\left ( Ax\wedge By+Bx\wedge Ay\right).\]
We will also denote by $\Id:\Lambda^2E\to\Lambda^2E$ the identity operator.
\begin{dfn}
  A curvature cone $\mathcal{C}\subset S^2_B\Lambda^2\mathbb{R}^n$ is said to satisfy condition $(\ast)$ if:
  \begin{itemize}
  \item $\mathcal{C}$ is Ricci flow invariant.
  \item If $\R\in\mathcal{C}$ then $\ric(\R)\wedge \id\subset\mathcal{C}$.
  \end{itemize}
\end{dfn}

The notion of Ricci flow invariant curvature cone is recalled in
appendix \ref{sec:curv-cones-ricci}. 

We mention here two examples of curvature cones which satisfy conditon
$(\ast)$ :
\begin{itemize}
\item The cone of curvature operators with
  positive isotropic curvature satisfies $(\ast)$. This
  follows from \cite{MR2587576} or \cite{MR2449060} and the two lemmas
  proven in the appendix. 
\item if $\mathcal{C}$ is Ricci flow invariant and $\{\R\geq 0\}\subset\mathcal{C}\subset\{\ric(\R)\geq 0\}$,
  then it satisfies condition $(\ast)$: indeed the eigenvalues of
  $\ric(\R)\wedge\id$ are $\tfrac{\lambda_i+\lambda_j}{2}$ where the
  $\lambda_i$'s are the eigenvalues of $\ric(\R)$ which are
  nonnegative by assumption. This includes the cones of
  nonnegative curvature operators, $2$-nonnegative curvature operators
  as well as the NNIC1 and NNIC2 cones. 
\end{itemize}

\begin{thm}\label{thm:mainthm}
  Let $\mathcal{C}$ be a curvature cone which satisfies condition
  $(\ast)$. Let $M$ be a compact smooth manifold and $\underline{u}$ be a
  lower semi-continuous nonnegative function on $M$. Let $g_i$ be a sequence of
  complete smooth metrics with bounded curvature on $M$ such that:
  \begin{itemize}
  \item As $i$ goes to infinity, $g_i$ converges in $C^0$ norm to a smooth metric $g$.
  \item For every $i$, the curvature operator of $g_i$ satisfies $\R_{g_i}-\underline{u}\Id\in\mathcal{C}$.
  \end{itemize}
  Then $\R_g -\underline{u}\Id\in\mathcal{C}$.
\end{thm}
Picking for $\mathcal{C}$ the NNIC cone or the cone of nonnegative
curvature operators, we recover Theorem \ref{thm1}

The proof of this result follows roughly Bamler's proof in
\cite{zbMATH06609371} with some differences:
\begin{itemize}
\item the evolution equation for the curvature operator along the
  Ricci flow is not as nice as the evolution equation for the scalar
  curvature, this is why we have not been able to handle the case of a
  lower bound $\underline{u}$ of arbitrary sign.
\item we require compactness of the manifold $M$ because we are not
  able to localize the argument as Bamler did. Our proof would work if
  we assumed completeness and bounded curvature of all the metrics
  involved, in order to be able to apply the maximum principle.
\item we study the behavior of the heat flow as $t$ goes to $0$ using
  hessian estimates based on the maximum principle, whereas Bamler
  uses heat kernel estimates. This choice has been made to make the
  proof more self-contained.
\end{itemize}

\subsubsection*{Acknowledgments.}
\label{sec:acknowledgments}

The author thanks Alix Deruelle for useful discussions during the
preparation of this paper.

\section{Bounding the curvature from below by the heat flow}
\label{sec:bound-curv-from}

\begin{prop}\label{prop:CurvHeatFlow}
  Let $\mathcal{C}$ be a curvature cone which satisfy condition
  $(\ast)$. Let $(M,g(t))_{t\in[0,T)}$ be a solution to the Ricci flow and
  $u(t,\cdot)$ be a nonnegative solution to the heat equation such that:
    \[\R_{g(0)}-u(0,\cdot)\Id\in\mathcal{C}\]
    Then
    \[\R_{g(t)}-u(t,\cdot)\Id\in\mathcal{C}\]
    for all $t\in[0,T)$
\end{prop}
\begin{proof}
  We will apply Hamilton's maximum principle (see \cite{MR862046}, Theorem
  4.3) to $\Lop(t)=\R_{g(t)}-u(t,\cdot)\Id$. It satisfies the
  evolution equation:
  \begin{align*}
    (\partial_t-\Delta)\Lop&=2Q(\R)\\
    &=2Q(\Lop)+4u\, Q(\Lop,\Id)+2u^2Q(\Id)\\
    &=2Q(\Lop)+4u\ric(\Lop)\wedge\id+2(n-1)u^2\Id
  \end{align*}
  where $Q(\R)=\R^2+\R^\#$ and we have used the notations and results
  from \cite{MR2415394}.

  $(n-1)u^2\Id\in\mathcal{C}$ since $\mathcal{C}$ is a curvature cone,
  $2u\ric(\Lop)\wedge\id\in\mathcal{C}$ whenever $\Lop\in\mathcal{C}$
  by condition $(\ast)$ and the fact that $u$ is nonnegative,
  $Q(\Lop)\in T_{\Lop}\mathcal{C}$ whenever $\Lop\in\mathcal{C}$ since
  $\mathcal{C}$ is Ricci flow invariant. Thus we have that
  $(\partial_t-\Delta)\Lop\in T_{\Lop}\mathcal{C}$ whenever
  $\Lop\in\mathcal{C}$. Hamilton's maximum principe implies that
  $\Lop(t)\in\mathcal{C}$ if $\Lop(0)\in\mathcal{C}$.
\end{proof}

\section{A variation on Koch and Lamm's Ricci flow of $C^0$ metrics.}
\label{sec:lower-bounds-curv}
We review here the theory of Ricci flow of $C^0$ metrics develloped by
Koch and Lamm in \cite{zbMATH06537677} and \cite{zbMATH06060679}. The
precise version of the theory we need here is formulated in terms of Ricci
flow background rather than a static background metric, and is very
similar to the theory developped by Deruelle and Lamm in
\cite{2016arXiv160100325D} formulated with an expanding Ricci soliton
as a background. 

Let $g(t)_{t\in[0,T)}$ be a Ricci flow on a compact manifold $M$. We
consider the following problem:
\begin{equation}
\begin{cases}
  \partial_t\bar{g}=-2\ric_{\bar{g}}+\mathcal{L}_{W(\bar{g},g)}\bar{g}\\
  \bar{g}(0)=\bar{g}_0
\end{cases}\tag{RicDT}\label{eq:RicciDeT}
\end{equation}
where $W(\bar{g},g)$ is the vector field given in local coordinates
by: $W^k=g^{ij}\left(\bar\Gamma_{ij}^k-\Gamma_{ij}^k\right)$ where
$\bar{\Gamma}_{ij}^k$ and $\Gamma_{ij}^k$ are the Christoffel symbols
of $\bar{g}(t)$ and $g(t)$ (respectively).

If we set $h=\bar{g}-g$, a computation shows that the equation above is equivalent
to:
\begin{align*}
  (\partial_t-\Delta_{L,g(t)})h =&\nabla_l\left (
    \left ((g+h)^{lm}-g^{lm}\right)\nabla_mh_{ik}\right)\\
  &+\bar{g}^{-1}\circledast
  \bar{g}^{-1}\circledast\nabla h\circledast\nabla h\\
  &+\left((g+h)^{-1}-g^{-1}\right)\circledast R_{g(t)} \circledast h\\
  =& \mathcal{F}[h]
\end{align*}
where $\nabla$ denotes the covariant derivative with respect to
$g(t)$, $R_{g(t)}$ is the curvature tensor of $g(t)$ 
and $\Delta_{L,g(t)}$ is the Lichnerowicz Laplacian with respect to
$g(t)$ (see \cite{MR2274812}, p. 109).
If $A$ and $B$ are tensors, $A \circledast B$ denotes any tensor built
by tracing $A\otimes B$.

Following \cite{zbMATH06060679}, \cite{zbMATH06537677} and
\cite{2016arXiv160100325D}, let $X_{T'}$ be the space of time dependent symmetric $2$-tensors $h$ which
satisfy:
\[\|h\|_{X}=\sup_{M\times [0,T')}|h|+\sup_{(x,R^2)\in M\times
             (0,T')}\left(R\|\bar\nabla h\|_{2,x,R}+\|\sqrt{t}\bar\nabla
             h\|^+_{n+4,x,R} \right)<+\infty \]
where:
\begin{align*}
\|h\|_{p,x,R}&=\left(
    \frac{\displaystyle\int_0^{R^2}\int_{B_{\bar{g}(s)}(x,R)}|h|^pdv_{\bar
               g(s)}ds}{\displaystyle\int_0^{R^2}\vol_{\bar
               g(s)}B_{\bar{g}(s)}(x,R)ds}\right)^{1/p}\\
\|h\|^+_{p,x,R}&=\left(
    \frac{\displaystyle\int_{\tfrac{1}{2}R^2}^{R^2}\int_{B_{\bar{g}(s)}(x,R)}|h|^pdv_{\bar
               g(s)}ds}{\displaystyle\int_{\tfrac{1}{2}R^2}^{R^2}\vol_{\bar
               g(s)}B_{\bar{g}(s)}(x,R)ds}\right)^{1/p}.
\end{align*}

For a fixed initial condition $h_0$ small enough in $C^0$ norm, one
can consider the map $\Phi$, 
defined on $X_{T'}$ for $T'$ small enough, which sends a small enough time dependent
symmetric $2$-tensor $k(t)$ to the solution $h(t)$ of the linear problem:
\[
\begin{cases}
  (\partial_t-\Delta_{L,g(t)})h =\mathcal{F}[k]\\
  h(0)=h_0.
\end{cases}
\]
\begin{lem}
  Provided $T'$ and the $C^0$ norm of $h_0$ are small enough,
  $\Phi:X_{T'}\to X_{T'}$ is a strict contraction from a small ball in
  $X_{T'}$ to itself.
\end{lem}

The proof
of this result follows the same route as the proof of Theorem 4.3 in
\cite{zbMATH06060679}, with adjustments needed to handle the fact that
the background metric is evolving by Ricci flow. These adjustments
have been carried out in details in \cite{2016arXiv160100325D}, though
our situation requires much less delicate estimates since we are only
interested in short time existence and uniqueness, whereas
\cite{2016arXiv160100325D} study the long time behavior of the solutions.

Banach's fixed point theorem can then be applied to show that $\Phi$
has a unique fixed 
point and get the following existence and uniqueness theorem:

\begin{thm}\label{thm:koch-lamm}
  Let $(M,g_0)$ be a compact smooth manifold with a smooth riemannian metric,
  let $(g(t))_{t\in[0,T)}$ be the smooth Ricci flow starting from $g_0$.

  Then there exists $\eps>0$ and $T'\in (0,T)$ such that if $\bar{g}_0$ is a metric with 
  $|\bar{g}_0-g_0|<\eps$ then there exists a solution $(\bar{g}(t))_{t\in[0,T')}$ to
  the equation:
  \[\partial_t\bar{g}(t)=-2\ric_{\bar{g}(t)}+\mathcal{L}_{W(\bar{g}(t),g(t))}\bar{g}(t)\]
  with $\bar{g}(0)=\bar{g}_0$.

  Moreover for every $k>0$,  there exist constants $C_k>0$ such that:
  \begin{equation}
    t^{k/2}\sup_M|\nabla^k(\bar{g}(t)-g(t))|\leq C_k\sup_M|\bar{g}_0-g_0|\label{eq:HigherOrder}  
  \end{equation}
  for any $t\in(0,T')$.

  The solution $\bar{g}(t)$ is unique among all solutions which satisfy
  the above estimate for $k=0$ and $k=1$.
\end{thm}

\section{Hölder continuity of the heat flow in a Ricci flow background}
\label{sec:hold-cont-heat}
In this section, we will show an estimate on solutions to the heat
equation which will be used to control how fast a solution can deviate
from its initial contion.

We start with the following well known lemma:
\begin{lem}\label{prop:Gradient}
  Let $(M,g(t))$ be a solution to the Ricci flow and $u(t)$ be a $C^2$
  solution to the heat equation $\partial_tu=\Delta_{g(t)}u$ then :
  \[\sup_M|\nabla u(t)|\leq\sup_M|\nabla u(0)|.\]
\end{lem}
\begin{proof}
  Once one notices that Bochner's formule implies 
  \[(\partial_t-\Delta_{g(t)})|\nabla u(t)|^2=-2|\nabla^2u(t)|^2\leq 0,\] 
  this is a straightforward consequence of the maximum principle.
\end{proof}
\begin{prop}\label{prop:Hessian}
  Let $(M,g(t))$ be Ricci flow such that $\sup_M|\R_{g(t)}|\leq A/t$
  for some $A>0$, let $u(t)$ be a $C^2$ solution to
  $\partial_tu=\Delta_{g(t)}u$. Then for every $\alpha>1/2$ there
  exists $T'(A,\alpha)\in (0,T)$ such that for every $t\in(0,T')$:
  \[\sup_M|\nabla^2u(t)|\leq \frac{\sup_M|\nabla u(0)|}{t^\alpha}.\]
\end{prop}
\begin{proof}
  We first compute, using the evolution equation for the Hessian given
  in \cite{MR2274812} Lemma 2.33:
  \begin{align*}
    (\partial_t-\Delta_{g(t)})|\nabla^2u|^2
    &=\ric_{g(t)}\ast\nabla^2u\ast\nabla^2u+2\ps{(\Delta_L-\Delta)\nabla^2u}{\nabla^2u}-2|\nabla^3u|^2\\ 
    & =\R_{g(t)}\ast \nabla^2u\ast\nabla^2u -2|\nabla^3u|^2\\ 
    & \leq \frac{C_nA}{t}|\nabla^2u|^2
  \end{align*}
  where $C_n$ is a dimensional constant.

  We now fix $\alpha>1/2$ and set $\delta=2\alpha-1>0$ and 
  $F(t)=t^{1+\delta}|\nabla^2u|^2+|\nabla u|^2$. $F$ satisfies:
  \[(\partial_t-\Delta_{g(t)})F\leq
    \left(C_nAt^{\delta}+(1+\delta)t^{\delta}-2\right)|\nabla^2 u|^2.\]
  The right hand side is negative for $t\leq
  \left(\tfrac{2}{C_nA+1+\delta}\right)^{1/\delta}=T'$. Thus the
  maximum principle implies that $\sup_MF(t)\leq
  \sup_MF(0)=\sup_M|\nabla u(0)|^2$ for $t\in (0,T')$. 

  Since
  $t^{1+\delta}|\nabla^2u(t)|^2\leq F(t)$, we have the required estimate.
\end{proof}

\begin{prop}\label{HolderHeatFlow}
  Let $\bar{g}(t)$ be a solution to equation \eqref{eq:RicciDeT} given
  by Theorem \ref{thm:koch-lamm}, and
  $\bar u$ be a $C^2$ solution to:
  \[\partial_t\bar u=\Delta_{\bar g(t)}\bar u+\ps{W}{\bar\nabla\bar{u}}.\]
  Then for every $\beta\in (0,\tfrac{1}{2})$, there exist constants
  $C,T'>0$ depending only on $\sup_M|\bar g_0-g_0|$ and
  $\sup_M|\bar\nabla\bar u(0)|$
  such that:
  \[\sup_M|\bar u(t)-\bar u(0)|\leq Ct^\beta\]
  for $t\in[0,T')$.
\end{prop}
\begin{proof}
  Let $W=W(\bar g,g)$ be the vector field built from the solution
  $\bar g(t)$ to equation \eqref{eq:RicciDeT} given by Theorem
  \ref{thm:koch-lamm}. Let us remark that $|\R_{\bar g(t)}|\leq A/t$
  and $|W(\bar g(t),g(t))|\leq B/\sqrt{t}$ thanks to the estimate
  \eqref{eq:HigherOrder}. 
  
  Let $\ph_t$ be the flow of the vector field
  $-W$. Then set $\tilde g(t)=\ph_t^*\bar g(t)$ and $\tilde
  u(t)=\ph_t^*\bar u(t)$. We have that:
  \[
  \begin{cases}
    \partial_t\tilde g=-2\ric_{\tilde{g}}\\
    \partial_t\tilde{u}=\Delta_{\tilde g}\tilde u
  \end{cases}
  \]
  We will have that $|\R_{\tilde g(t)}|\leq A/t$ since the same
  estimate was true for $\bar g(t)$.  Thus we can apply Propositions
  \ref{prop:Gradient} and \ref{prop:Hessian} to get that:
  \[\sup_M|\tilde\nabla\tilde u|\leq C_1,\ \sup_M|\tilde\nabla^2\tilde
  u|\leq \frac{C_2}{t^\alpha}\]
  for any $\alpha >\tfrac{1}{2}$ and $t\leq T(\alpha,A)$. Those same
  estimates will thus hold for $\bar u$. 

  We can then write:
  \begin{align*}
    |\bar{u}(t)-\bar{u}(0)|\leq&\int_0^t|\Delta_{\bar{g}} \bar u|+|\ps{W}{\bar\nabla\bar
    u}|dt\\
  \leq &\int_0^t\sqrt{n}|\bar\nabla^2 \bar u|+|W|\,|\bar\nabla\bar
    u|dt\\
    \leq&\int_0^t\sqrt{n}C_2t^{-\alpha}+C_1Bt^{-1/2}\, dt\\
    \leq&Ct^{1-\alpha}.
  \end{align*}
\end{proof}
\section{Proof of the main result}
\label{sec:proof-main-result}

\begin{proof}[Proof of Theorem \ref{thm:mainthm}]
  In this section we prove Theorem \ref{thm:mainthm}.

  For any lower semi continuous $\underline u:M\to\mathbb{R}$, one can
  find a sequence of $v_k$ of smooth functions such that $v_k\leq
  \underline{u}$ and $v_k$ pointwisely converges to $\underline{u}$.
  Thus we can assume that $\underline{u}:M\to\mathbb{R}$ is
  $C^\infty$.

  Let $g(t)$ be the Ricci flow of $g$ and $g_i(t)$ be the solution to
  equation \eqref{eq:RicciDeT} such that $g_i(0)=g_i$.

  Thanks to Theorem \ref{thm:koch-lamm}, we have that $g_i(t)$
  converges in $C^{\infty}_{loc}((0,T)\times M)$ to $g(t)$. Set
  $u_i(t)$ be the solution to:
  \[\partial_t u_i=\Delta_{g_i(t)}u_i+\ps{W(g(t),g_i(t))}{\nabla
    u_i}\] with $u_i(0)=\underline u$.

  Thanks to the maximum principle $u_i(t)$ is bounded uniformly in
  $i\in \mathbb{N}$. We also have uniform bounds on the Hessian and
  the gradient of each $u_i(t)$ thanks to the results of section
  \ref{sec:hold-cont-heat}. 
  Thus, up to a subsequence, $u_i(t)$
  converges locally uniformly on $(0,T)\times M$ to a solution $u(t)$ of the equation:
  \[\partial_t u=\Delta_{g(t)}u.\]

  Thanks to proposition \ref{HolderHeatFlow}, we have that for each
  $i$, $\sup_M| u_i(t)-\underline{u}|\leq Ct^{1/4}$. Hence $\sup_M|
  u(t)-\underline{u}|\leq Ct^{1/4}$ and $u(t)$ uniformly converges to
  $\underline u$ as $t$ goes to $0$.

  Moreover, up to a pull back by a time dependent diffeomorphism,
  $g_i(t)$ and $u_i(t)$ satisfy the hypothesis of proposition
  \ref{prop:CurvHeatFlow}. Hence we have that
  $\R_{g_i(t)}-u_i(t)\Id\in\mathcal C$.

  Since $g_i(t)$ converges in $C^2$ to $g(t)$ for every fixed $t>0$,
  we have that :
  \[\R_{g(t)}-u(t)\Id\in\mathcal C.\]

  We now let $t$ go to $0$ to get :
  \[\R_{g}-\underline{u}\Id\in\mathcal C.\]

  This ends the proof of Theorem \ref{thm:mainthm}.
\end{proof}
\appendix

\section{Curvature cones and the Ricci flow}
\label{sec:curv-cones-ricci}

We gather here some definitions on Ricci flow invariant curvature
cones for convenience of the reader. For a more detailed exposition of
this topic see \cite{Richard2012-2014}.

Recall that the space of algebraic curvature operators, denoted by $S^2_B\Lambda^2\mathbb{R}^n$, is the space of symmetric
endomorphisms $\R:\Lambda^2\mathbb{R}^n\to\Lambda^2\mathbb{R}^n$ which
satisfy the Bianchi identity: 
\[\ps{\R(x\wedge y)}{z\wedge t}+\ps{\R(y\wedge z)}{x\wedge t}+\ps{\R(z\wedge x)}{y\wedge t}=0.\]

The orthogonal group $O(n)$ acts on $S^2_B\Lambda^2\mathbb{R}^n$ by:
\[\ps{g\cdot\R(x\wedge y)}{z\wedge t}=\ps{\R(gx\wedge gy)}{gz\wedge gt}.\]
\begin{dfn}
  A curvature cone is closed convex cone $\mathcal{C}\subset
  S^2_B\Lambda^2\mathbb{R}^n$ which is $O(n)$ invariant and contains
  the identity operator
  $\Id:\Lambda^2\mathbb{R}^n\to\Lambda^2\mathbb{R}^n$ in its interior.
\end{dfn}
Thanks to the $O(n)$ invariance of $\mathcal{R}$, it makes thanks to
see $\mathcal{C}$ as a subset of $S^2_B\Lambda^2T_xM$ for each point
$x$ in a Riemannian manifold $(M,g)$, and thus it makes sense to say
that the curvature operator $\R_g$ of a Riemannian manifold $(M,g)$
belongs to $\mathcal{C}$.

Recall that, once Uhlenbeck's trick has been applied, the curvature
operator $\R_{g(t)}$ of a Ricci flow $(M,g(t))$ satisfies
\[\partial_t\R_{g(t)}=\Delta_{g(t)}\R_{g(t)}+Q(\R_{g(t)})\]
where $Q(\cdot)$ can be seen as a quadratic vector field
$Q:S^2_B\Lambda^2\mathbb{R}^n\to S^2_B\Lambda^2\mathbb{R}^n$. We will
write $Q(\cdot,\cdot)$ for the associated bilinear map.

With Hamilton's tensor maximum principle in mind, we have the
following definition:
\begin{dfn}
  A curvature cone $\mathcal{C}$ is said to be Ricci flow invariant if
  $Q(\R)\in T_{\R}\mathcal{C}$ whenever $\R\in\mathcal{C}$.
\end{dfn}

Examples of Ricci flow invariant curvature cones include the cones of
curvature operator which are nonnegative or $2$-nonnegative as
symmetric quadratic forms, the cone of NNIC curvature operators and the
related NNIC1 and NNIC2 cones.

\section{Isotropic curvature and Ricci curvature}
\label{sec:isotr-curv-ricci}
If $\R$ is a curvature operator, we set $\R_{ijkl}=\ps{\R(e_i\wedge
  e_j)}{e_k\wedge e_l}$. Recall that $\R$ is said to have nonnegative
isotropic curvature (in short NNIC) if, for any orthonormal $4$-frame
$(e_1,e_2,e_3,e_4)$, we have:
\[\mathcal{IC}_{1234}(\R)=\R_{1313}+\R_{1414}+\R_{2323}+\R_{2424}-2\R_{1234}>0.\]

Recall that a symmetric tensor is said to be $k$-nonnegative if the
sum of its $k$ smallest eigenvalues is nonnegative.

These two lemmas show that that the cone of curvature with nonnegative
isotropic curvature satisfy condition $(\ast)$.
\begin{lem}
  Let $\R$ be a NNIC curvature operator, then $\ric(\R)$ is $4$-nonnegative.
\end{lem}
\begin{proof}
  Let $(e_i)_{i=1,\dots,n}$ be any orthonormal basis of
  $\mathbb{R}^n$.

  If $\R$ is PIC, we have:
  \[\R_{1i1i}+\R_{1j1j}+\R_{2i2i}+\R_{2j2j}-2\R_{12ij}\geq 0.\]
  Summing this and the same expression obtained by exchanging $e_i$ and
  $e_j$, we get:
  \[\R_{1i1i}+\R_{1j1j}+\R_{2i2i}+\R_{2j2j}\geq 0.\]
  We sum together the $n-3$ terms corresponding to letting $i$ ranging
  from $3$ to $n$, excluding $j$, we get:
  \[\R_{11}+\R_{22}+(n-4)(\R_{1j1j}+\R_{2j2j})-2\R_{1212}\geq 0\]
  where $\R_{ii}$ stands for $\ric(\R)(e_i,e_i)$.

  We now sum over $j$ ranging from $3$ to $n$. This gives:
  \[(n-3)(\R_{11}+\R_{22}-2\R_{1212})\geq 0.\]
  Thus we have $\R_{11}+\R_{22}\geq 2\R_{1212}$. Hence:
  \begin{align*}
    \R_{11}+\R_{22}+\R_{33}+\R_{44}
    =&\tfrac{1}{2}(\R_{11}+\R_{33})+\tfrac{1}{2}(\R_{11}+\R_{44})\\ 
    &+\tfrac{1}{2}(\R_{22}+\R_{33})+\tfrac{1}{2}(\R_{22}+\R_{44})\\ 
    \geq &\R_{1313}+\R_{1414}
    +\R_{2323}+\R_{2424}\\ \geq &0.
  \end{align*}
  Since this inequality is true for any orthonormal $4$-frame
  $(e_1,e_2,e_3,e_4)$, we have that $\ric(\R)$ is $4$-positive.
\end{proof}
\begin{lem}
  Let $A$ be a $4$-nonnegative symmetric endomorphism, then $A\wedge\id$
  is NNIC.
\end{lem}
\begin{proof}
  Recall that if $(e_i,e_j,e_k,e_l)$ come from an orthonormal frame, we have that:
  \[(A\wedge\id)_{ijkl}=\frac{1}{2}\left
      (A_{ik}\delta_{jl}-A_{il}\delta_{jk}+\delta_{ik}A_{jl}-\delta_{il}A_{jk}\right)\]
  which implies that $(A\wedge\id)_{ijkl}=0$ if $i,j,k,l$ are all
  distinct and that
  $(A\wedge\id)_{ijij}=\tfrac{1}{2}\left(A_{ii}+A_{jj}\right)$ if
  $i\neq j$.

  Let $(e_1,\dots,e_4)$ be any orthonormal $4$-frame, and $A$ be a
  symmetric endomorphism. Then:
  \[\mathcal{IC}_{1234}(A\wedge\id)=A_{11}+A_{22}+A_{33}+A_{44}\geq 0\]
  since $A$ is $4$-positive.
\end{proof}
\bibliography{C0stabilityPIC.bib}

\begin{thebibliography}{{Bam}16}

\bibitem[{Bam}16]{zbMATH06609371}
Richard~H. {Bamler}.
\newblock {A Ricci flow proof of a result by Gromov on lower bounds for scalar
  curvature.}
\newblock {\em {Math. Res. Lett.}}, 23(2):325--337, 2016.

\bibitem[BBI01]{MR1835418}
Dmitri Burago, Yuri Burago, and Sergei Ivanov.
\newblock {\em A course in metric geometry}, volume~33 of {\em Graduate Studies
  in Mathematics}.
\newblock American Mathematical Society, Providence, RI, 2001.

\bibitem[BDS17]{2017arXiv170300984B}
J.~{Basilio}, J.~{Dodziuk}, and C.~{Sormani}.
\newblock {Sewing Riemannian Manifolds with Positive Scalar Curvature}.
\newblock {\em ArXiv e-prints}, March 2017.

\bibitem[BS09]{MR2449060}
Simon Brendle and Richard Schoen.
\newblock Manifolds with {$1/4$}-pinched curvature are space forms.
\newblock {\em J. Amer. Math. Soc.}, 22(1):287--307, 2009.

\bibitem[BW08]{MR2415394}
Christoph B\"ohm and Burkhard Wilking.
\newblock Manifolds with positive curvature operators are space forms.
\newblock {\em Ann. of Math. (2)}, 167(3):1079--1097, 2008.

\bibitem[CLN06]{MR2274812}
Bennett Chow, Peng Lu, and Lei Ni.
\newblock {\em Hamilton's {R}icci flow}, volume~77 of {\em Graduate Studies in
  Mathematics}.
\newblock American Mathematical Society, Providence, RI; Science Press Beijing,
  New York, 2006.

\bibitem[DL16]{2016arXiv160100325D}
A.~{Deruelle} and T.~{Lamm}.
\newblock {Weak stability of Ricci expanders with positive curvature operator}.
\newblock {\em ArXiv e-prints}, January 2016.

\bibitem[Gro14]{MR3201312}
Misha Gromov.
\newblock Dirac and {P}lateau billiards in domains with corners.
\newblock {\em Cent. Eur. J. Math.}, 12(8):1109--1156, 2014.

\bibitem[Ham86]{MR862046}
Richard~S. Hamilton.
\newblock Four-manifolds with positive curvature operator.
\newblock {\em J. Differential Geom.}, 24(2):153--179, 1986.

\bibitem[KL12]{zbMATH06060679}
Herbert {Koch} and Tobias {Lamm}.
\newblock {Geometric flows with rough initial data.}
\newblock {\em {Asian J. Math.}}, 16(2):209--235, 2012.

\bibitem[KL15]{zbMATH06537677}
Herbert {Koch} and Tobias {Lamm}.
\newblock {Parabolic equations with rough data.}
\newblock {\em {Math. Bohem.}}, 140(4):457--477, 2015.

\bibitem[LV09]{MR2480619}
John Lott and C\'edric Villani.
\newblock Ricci curvature for metric-measure spaces via optimal transport.
\newblock {\em Ann. of Math. (2)}, 169(3):903--991, 2009.

\bibitem[Ngu10]{MR2587576}
Huy~T. Nguyen.
\newblock Isotropic curvature and the {R}icci flow.
\newblock {\em Int. Math. Res. Not. IMRN}, (3):536--558, 2010.

\bibitem[Ric14]{Richard2012-2014}
Thomas Richard.
\newblock Curvature cones and the ricci flow.
\newblock {\em Séminaire de théorie spectrale et géométrie}, 31:197--220,
  2012-2014.

\bibitem[Stu05]{MR2118836}
Karl-Theodor Sturm.
\newblock Convex functionals of probability measures and nonlinear diffusions
  on manifolds.
\newblock {\em J. Math. Pures Appl. (9)}, 84(2):149--168, 2005.

\end{thebibliography}
\bibliographystyle{alpha}
\end{document}